\documentclass[12pt]{amsart}

\usepackage{a4}
\usepackage[T1]{fontenc}
\usepackage[utf8]{inputenc}
\usepackage[T1]{fontenc}
\usepackage[utf8]{inputenc}
\usepackage[italian, english]{babel}
\usepackage{amsmath}
\usepackage{amssymb}
\usepackage{amsthm}
\usepackage{mathrsfs}
\usepackage{tikz}
\usepackage{subfig}
\usepackage{newlfont}
\usepackage{graphicx}
\usepackage{fancyhdr}
\usepackage{indentfirst}
\usepackage{placeins}
\usepackage{eurosym}
\usepackage{subfig}
\usepackage{url}
\usepackage{newlfont}
\usepackage{xcolor}
\usepackage{listings}
\usepackage{quoting}
\usepackage{braket}
\usepackage[all]{xy}
\theoremstyle{plain}
\numberwithin{equation}{section}
\newtheorem{theorem}{Theorem}[section]
\theoremstyle{definition}
\newtheorem{definition}[theorem]{Definition}
\theoremstyle{plain}
\newtheorem{proposition}[theorem]{Proposition}
\theoremstyle{plain}
\newtheorem{lemma}[theorem]{Lemma}
\theoremstyle{plain}
\newtheorem{corollary}[theorem]{Corollary}
\theoremstyle{definition}

\theoremstyle{remark}

\theoremstyle{remark}

\newcommand{\numberset}{\mathbb}
\newcommand{\N}{\numberset{N}}

\DeclareMathOperator{\C}{C}
\DeclareMathOperator{\Def}{Def}

\usepackage{bm}
\newcommand{\ab}[1]{{\mathbf{#1}}}

\newcommand{\algop}[2]{( {#1}, {#2} )}
\newcommand{\ari}[1]{^{[\,#1\,]}}
\newcommand{\arii}[1]{^{[#1]}}

\title[Clonoids in universal algebraic geometry]{A clonoid based approach to
  some finiteness results in universal algebraic geometry}

\author{Erhard Aichinger}
\address{
Institut für Algebra,
Johannes Kepler Universit\"at Linz, Altenberger Strasse 69, 4040 Linz,
Austria}
\email{erhard@algebra.uni-linz.ac.at}

\author{Bernardo Rossi}
\address{Dipartimento di Ingegneria dell'informazione e scienze matematiche,
  Universit\`a degli Studi di Siena,
  San Niccol\`o, Via Roma 56, 53100 Siena, Italy}
\address{
Institut für Algebra,
Johannes Kepler Universit\"at Linz, Altenberger Strasse 69, 4040 Linz,
Austria}
\email{bernardo.rossi@student.unisi.it}

\thanks{Supported by the Austrian Science Fund (FWF):~P29931.}

\begin{document}
\maketitle
\begin{abstract}
  We prove that for a finite first order structure $\mathbf{A}$ and a set of first order formulas $\Phi$ in its language with certain closure properties, the finitary relations on $A$ that are definable via formulas in $\Phi$ are uniquely determined by those of arity $|A|^{2}$. This yields
  new proofs for some finiteness results from universal algebraic geometry.
 \end{abstract}
\section{Introduction}
To every algebraic structure $\ab{A}$, one can associate certain
subsets of
its finite direct powers $\ab{A}^n$ ($n \in \N$). For example, the \emph{relational clone} of $\ab{A}$
consists of all subuniverses of these direct powers. Two algebras defined on the same base
set $A$, but possibly with different basic operations, may have the same relational clone: if
$A$ is finite, this happens if and only if the two algebras are \emph{term equivalent}, i.e.,
each fundamental operation of one algebra is a term operation of the other algebra
\cite[p. 55, Folgerung~1.2.4]{PK:FUR}.
Universal algebraic geometry associates with $\ab{A}$ all solution sets $S$ of (possibly infinite)
systems of algebraic equations of the form
$S = \{ (a_1, \ldots, a_n) \in A^n \mid  f_i (a_1, \ldots, x_n) = g_i (a_1, \ldots, a_n) \text{ for all }
i \in I\}$, where $f_i, g_i$ are term operations from $\ab{A}$.
We will call such a solution set $S$
an \emph{algebraic set}. Following A.\ G.\ Pinus \cite[p.\ 501]{pinus17},
two algebras defined on the same base set $A$ are called \emph{algebraically equivalent} if
they have the same algebraic sets. An algebra is called an \emph{equational domain} if for
all $n \in \N$ and all algebraic subsets $S,T$ of $A^n$, the union $S \cup T$ is again algebraic.
Using a description of algebraic equivalence through certain invariants, Pinus proved
that on a finite set, there are at most finitely many algebraically inequivalent equational
domains \cite[Theorem~3]{pinus17}. In the present note, we observe that this finiteness can also
be obtained by considering the clonoid of the characteristic functions of algebraic sets
and applying a consequence of the Baker Pixley Theorem \cite[Theorem~2.1]{BP:PIAT} that
was recently proved by A.\ Sparks \cite[Theorem~2.1]{Sp:OTNO} to these clonoids.

\section{Definable sets}

We note that the solutions of a term equation $s(x_1, \ldots, x_n) = t (x_1, \ldots, x_n)$
in $n$ variables over the algebra $\ab{A}$ can be written in the form
$S = \{ (a_1, \ldots, a_n) \in A^n \mid \ab{A} \models \varphi (a_1, \ldots, a_n)\}$,
where $\varphi$ is a logical formula of the form $s \approx t$ with free 
variables $x_1, \ldots, x_n$, and $\varphi (a_1, \ldots, a_n)$ is the interpretation
of this formula in $\ab{A}$ with the variable assignment $x_i \mapsto a_i$. 
We put this into a more flexible frame allowing for arbitrary first order structures
with both functional and relational symbols.
By a \emph{first order formula}, we always understand a formula in first order logic
with
equality, denoted by $\approx$, over the set of variables $\{x_i \mid i \in \N\}$.

\begin{definition}%
  Let $\mathbf{A}=(A,  (f_{i})_{i\in I}, (\rho_{j})_{j\in J})$ be a first order structure,
  let $\Phi$ be a set of first order formulas in the language of $\mathbf{A}$,
  let $n \in \N$, and
  let $B \subseteq A^n$. $B$ is called \emph{$\Phi$-definable} if
  there is a formula $\varphi \in \Phi$ whose free variables are all contained in
  $\{x_{1},\dots, x_{n}\}$ such that
   \[
       B= \{(a_{1},\dots,a_{n})\in A^{n} \mid \ab{A}\models \varphi(a_{1},\dots,a_{n})\}.
   \]
   We define $\Def^{[n]}(\mathbf{A},\Phi)$ to be the set of all $\Phi$-definable subsets of $A^{n}$,
   and we set $\Def(\mathbf{A},\Phi):= \bigcup_{n \in \N} \Def^{[n]}(\mathbf{A},\Phi)$.
\end{definition}
For example, given a finite relational structure $\ab{A} = \algop{A}{(\rho_i)_{i \in I}}$,
and taking $\Phi$ to be the set of primitive positive formulas in the language
of $\ab{A}$
(allowing also the the binary equality symbol $\approx$), $\Def (\ab{A}, \Phi)$ is the
relational clone generated by $\{\rho_i \mid i \in I\}$ \cite[Hauptsatz~2.1.3(i)]{PK:FUR}. 
Other examples for the operator $\Def$ come from Universal Algebraic Geometry
(cf. \cite{Pi:ASOU, DMR:AGOA}).
 Let $\mathbf{A}=\algop{A}{(f_i)_{i \in I}}$ be a finite algebraic structure,
  and let $\Phi$ consist of all (finite) conjunctions of atomic formulas
  in the language of $\ab{A}$. Then each formula in $\Phi$ is of the form
  $\bigwedge_{i=1}^m s_i (x_1, \ldots, x_n) \approx t_i (x_1, \ldots, x_n)$, where $m, n \in \N$
  and
  all $s_i$ and $t_i$ are terms in the language of $\ab{A}$, and hence $\Def(\mathbf{A},\Phi)$ consists of all 
  algebraic sets of $\ab{A}$. Another collection of finitary relations
  on an algebra $\ab{A}$ that can be expressed in this setting is 
  the \emph{$L_0$-logical geometry of $\mathbf{A}$}, which was studied,
  e. g., in \cite{pinus17logeq}. This $L_0$-logical geometry is
  $\Def (\mathbf{A}, \Phi')$, where $\Phi'$ is the set of all quantifier
  free formulas of $\ab{A}$.
  Two algebras $\ab{A}_1$ and $\ab{A}_2$ defined on the same universe
  are called \emph{$L_0$-logically
    equivalent} if $\Def(\mathbf{A}_{1}, \Phi_1') = \Def (\mathbf{A}_{2}, \Phi_2')$, where
  $\Phi_i'$ is the set of quantifier free formulas in the language of $\ab{A}_i$.
 
  All sets of formulas that we have considered so far were closed under \emph{taking minors}. To define this concept, we
  use the operation of \emph{substitution} as defined in
  \cite[Definition~8.2]{EF:EIDM}.
  We call a first order formula $\varphi$ a \emph{minor of the formula} $\varphi'$ if
  there is an $n \in \N$ and a mapping $\sigma : \{1,\ldots, n\} \to \N$ such that
  $\varphi = \varphi'\frac{x_{\sigma(1)},\dots,x_{\sigma(n)}}{x_{1},\dots, x_{n}}$;
  in other words,
  $\varphi$ is the result of substituting $x_1, \ldots, x_n$ in $\varphi'$ simultaneously
  by $x_{\sigma(1)}, \ldots, x_{\sigma(n)}$, sometimes denoted
  by $\varphi = \varphi' (x_{\sigma(1)}, \ldots, x_{\sigma(n)})$.
  A function $f:A^{m}\rightarrow B$ is a \emph{minor of the function} $f':A^{n}\rightarrow B$ if there
  exists $\sigma:\{1,\dots,n\}\rightarrow \{1,\dots, m\}$ such that
  $f(x_{1},\dots,x_{m})=f'(x_{\sigma(1)},\dots,x_{\sigma(n)})$ for all $x_1, \ldots, x_m \in A$.
  Finally, a set $B\subseteq A^{m}$ is a \emph{minor of the set} $B'\subseteq A^{n}$
  if there exists $\sigma:\{1,\dots,n\}\rightarrow \{1,\dots, m\}$ such that
  $B=\{ (a_{1},\dots,a_{m})\in A^{m}\mid (a_{\sigma(1)},\dots,a_{\sigma(n)})\in B'\}$.
  We say that a subset $\mathcal{R}$ of the set $\bigcup_{n \in \N} \mathcal{P} (A^n)$ of all
  finitary relations on $A$ is \emph{closed under
  finite intersections} if for all $m \in \N$ and for all $S, T \subseteq A^m$ with
  $S \in \mathcal{R}$ and $T \in \mathcal{R}$, we have $S \cap T \in \mathcal{R}$; being
  \emph{closed under
  finite unions} is defined similarly.
  Using the substitution lemma for first order logic
  \cite[Substitutionslemma~8.3]{EF:EIDM}, we obtain:
\begin{proposition}\label{prop:connection_M_Sc}
Let $\mathbf{A}$ be a first order structure, and let $\Phi$ be a set of first order formulas in its language closed under $\wedge$, $\vee$, and taking minors of formulas. Then $\Def(\mathbf{A},\Phi)$ is closed under finite intersections, finite unions, and taking minors of sets.
\end{proposition}
For a subset $T$ of $A^n$, we define its \emph{characteristic function}
$\mathbf{1}_{T}\colon A^{n}\rightarrow\{0,1\}$ by $\mathbf{1}_T (x_1, \ldots, x_n) = 1$ if
$(x_1, \ldots, x_n) \in T$, and $\mathbf{1}_T (x_1, \ldots, x_n) = 0$ if $(x_1, \ldots, x_n) \not\in T$.
Let $A$ be a set, and let $\ab{B}$ be an algebra. Following \cite{AM:FGEC},
a subset $C$ of $\bigcup_{n \in \N} B^{A^n}$ is called a \emph{clonoid} from $A$ to $\ab{B}$ if
$C$ is closed under taking minors of functions, and for every $k \in \N$, the set $C^{[k]} := C \cap B^{A^k}$ is a subuniverse
of $\ab{B}^{A^k}$. If a set $S$ is a minor of the set $T$, then
its characteristic function $\mathbf{1}_S$ is a minor of the function $\mathbf{1}_T$.
Hence we have:
\begin{proposition}\label{prop:closed_under_Sc_union_intersection_implies_clonoid}
  Let $\mathbf{A}$ be a finite set, and let $\mathcal{R}$ be a set of finitary relations on $A$
  that is closed under finite intersections, finite unions, and under taking minors of sets.
  Then the set $\C(\mathcal{R}):=\{\mathbf{1}_{T}\mid T\in \mathcal{R}\}$ is a clonoid
  from $A$ to the two element lattice $(\{0,1\},\wedge,\vee)$.
\end{proposition}
We call $\C (\mathcal{R})$ the \emph{characteristic clonoid} of $\mathcal{R}$.

\section{Recovering definable sets from those of bounded arity}
Our main tool is a consequence of A.\ Sparks's description of clonoids
from a finite set into an algebra with a near-unanimity term. We note
that a lattice has the near-unanimity (or majority) term
$(x \land y) \lor (x \land z) \lor (y \land z)$, and hence
as a consequence of \cite[Theorem~2.1]{Sp:OTNO}, a clonoid from a finite set $A$ into
the two element lattice is generated by its $|A|^2$-ary members.
\begin{lemma}[{\cite[Theorem~2.1]{Sp:OTNO}}] \label{teor:clonoids_charcherized_by_A_square}
  Let $A$ be a finite set, let $\ab{B} := (\{0,1\},\wedge, \vee)$ be the two element lattice,
  and let $C,D$ be two clonoids from $A$ to $\ab{B}$.
  Then $C=D$ if and only if $C\ari{|A|^{2}}=D\ari{|A|^{2}}$. 
\end{lemma}
We apply this result to the characteristic clonoids of some
$\Phi$-definable sets and obtain:
\begin{theorem}\label{teor:fin_many_class_closed_under_unio_int_Sc}
  Let $\mathbf{A}_{1}$  and $\mathbf{A}_{2}$ be two first order structures on a finite set $A$.
  For each $i \in \{1,2\}$, let $\Phi_{i}$ be a set of first order formulas in the language of $\mathbf{A}_{i}$
  that is closed under $\wedge$, $\vee$, and taking minors of formulas. 
Then $\Def(\mathbf{A}_{1},\Phi_{1})=\Def(\mathbf{A}_{2},\Phi_{2})$ if and only if $\Def\ari{|A|^{2}} (\mathbf{A}_{1},\Phi_{1})=\Def\ari{|A|^{2}} (\mathbf{A}_{2},\Phi_{2})$.  
\end{theorem}
\begin{proof}
  Let $\mathcal{R}:=\Def(\mathbf{A}_{1},\Phi_{1})$ and $\mathcal{S}:=\Def(\mathbf{A}_{2},\Phi_{2})$.
  By  Proposition~\ref{prop:connection_M_Sc}, $\mathcal{R}$ and $\mathcal{S}$ are
  closed under finite intersections, finite unions, and taking minors of sets. Therefore,
  by Proposition~\ref{prop:closed_under_Sc_union_intersection_implies_clonoid},
  their characteristic clonoids $C (\mathcal{R})$ and $C(\mathcal{S})$ are clonoids
  from $A$ into the two element lattice.
  Clearly, $\mathcal{R}=\mathcal{S}$ if and only if $\C(\mathcal{R})=\C(\mathcal{S})$.
  By Lemma~\ref{teor:clonoids_charcherized_by_A_square}, the last equality holds
  if and only if $\C(\mathcal{R})\ari{|A|^{2}}=\C(\mathcal{S})\ari{|A|^{2}}$,
  which is equivalent to $\Def\ari{|A|^{2}} (\ab{A}_1, \Phi_1) =
                          \Def\ari{|A|^{2}} (\ab{A}_2, \Phi_2)$.
\end{proof}
From this result, it follows that
an arbitrary set of formulas closed under $\wedge, \vee$, and taking minors can
sometimes be replaced by a set of equal ``expressive power'' that
consists only of disjunctions of conjunctions
of atomic formulas:
\begin{corollary}
  Let $\mathbf{A}$ be a finite first order structure, let $\Phi$ be a set of first order formulas in the language of $\mathbf{A}$ that is  closed under $\wedge$, $\vee$, and taking minors,
  and let $m := |A|^2$. We choose a finite set $I$ and a family $(S_i)_{i \in I}$ of subsets
  of $A^m$ such that
    $\{S_{i}\mid i\in I\} = \Def\arii{m}(\mathbf{A},\Phi)$,  and we 
let $(\sigma_{i})_{i\in I}$ be a family of relational symbols of arity $m$.
For each $i \in I$, let $\sigma_i^{\ab{A}'} := S_i$, and let
$\ab{A}'$ be the relational structure $\algop{A}{(\sigma_{i}^{\ab{A}'})_{i\in I}}$.
Let $\Psi$ be the closure of $\{\sigma_{i} (x_1,\ldots, x_m) \mid i\in I\}$ under taking minors of formulas, $\wedge$,
and $\vee$.
Then
\(
\Def(\mathbf{A},\Phi)=\Def(\ab{A}', \Psi).
\) 
\end{corollary}
\begin{proof}
  We first show
  \begin{equation} \label{eq:psi}
    \Def\arii{m} (\ab{A}', \Psi) = \{S_i \mid i \in I\}.
  \end{equation}
  For $\supseteq$, we observe that $S_i$ is definable
  by the formula $\sigma_i (x_1,\ldots, x_m)$.
  For $\subseteq$, we choose a set $T \subseteq A^m$
  that is definable by $\psi \in \Psi$ with
  free variables $x_1, \ldots, x_m$. Then $\psi$ is equivalent
  to a formula $\psi' = \bigwedge_{k \in K} \bigvee_{l \in L} \sigma_{i (k,l)} (x_{\tau (k,l,1)},\ldots, x_{\tau (k,l,m)})$ with $i(k,l) \in I$ and $\tau (k,l,r) \in \{1,\ldots, m\}$ for all
  $k \in K$, $l \in L$, and $r \in \{1,\ldots, m\}$.
  For each $i \in I$, $S_i \in \Def\arii{m} (\ab{A}, \Phi)$, and therefore
  there is a formula $\varphi_i \in \Phi$ such that
  $S_i$ is definable by $\varphi_i$.
  Now $$\varphi' := \bigwedge_{k \in K} \bigvee_{l \in L} \varphi_{i (k,l)} (x_{\tau (k,l,1)},\ldots,
  x_{\tau (k,l,m)})$$ is a formula in $\Phi$ that defines $T$.
  Hence $T \in \Def\arii{m} (\ab{A}, \Phi) = \{S_i \mid i \in I\}$, which completes
  the proof of~\eqref{eq:psi}. Therefore
  $\Def\arii{m} (\ab{A}', \Psi) = \Def\arii{m} (\ab{A}, \Phi)$. Applying
  Theorem~\ref{teor:fin_many_class_closed_under_unio_int_Sc}
  to $\ab{A}_1 := \ab{A}$, $\Phi_1 := \Phi$, $\ab{A}_2 := \ab{A}'$,
  $\Phi_2 := \Psi$, we obtain $\Def(\mathbf{A},\Phi)=\Def(\ab{A}', \Psi)$.
\end{proof}

\section{Inequivalent algebras}
Some of the finiteness results from universal algebraic geometry
from \cite{pinus17, pinus17logeq} can be viewed as consequences
of Theorem \ref{teor:fin_many_class_closed_under_unio_int_Sc}.
Let $\ab{A}$ be an algebraic structure.
It is easy to see the the set of first order formulas consisting of all primitive positive formulas and the set of all finite conjunctions of atomic formulas in the language of $\ab{A}$ are both closed under $\wedge$ and taking minors, but not under $\vee$. Therefore it is impossible to apply Theorem \ref{teor:fin_many_class_closed_under_unio_int_Sc} to all relational clones or to the algebraic geometry of every finite
universal algebra. Hence we first restrict ourselves to \emph{equational domains}; in these algebras,
the union
of two algebraic sets is again algebraic.
For example, every field is an equational domain.
\begin{corollary}[{\cite[Theorem~3]{pinus17}}]
  Let $A$ be a finite set, and let
  $(\ab{A}_i)_{i \in I}$ be a family of pairwise algebraically inequivalent
  equational domains with universe $A$. Then $I$ is finite. 
\end{corollary}
\begin{proof}
  Let $m := |A|^2$.
  For each $i \in I$, we  let $\Phi_i$ be the set of finite conjunctions of atomic formulas
  of $\ab{A}_i$, and we let $\Phi_i'$ be the smallest set of
  formulas that contains $\Phi_i$ and is closed under $\land$ and $\lor$.
  Since $\ab{A}_{i}$ is an equational domain, we can use induction on
  the number of $\lor$ in a formula $\varphi' \in \Phi_i'$ to show
  that the subset defined by $\varphi'$ is algebraic and hence
  lies in $\Def(\mathbf{A}_{i}, \Phi_i)$. Hence $\Def(\mathbf{A}_{i}, \Phi_i) = \Def(\mathbf{A}_{i}, \Phi_i')$.
  We will now show that the mapping
  $\alpha : I \to \mathcal{P} (\mathcal{P} (A^m))$,
  $\alpha(i) := \Def\arii{m} (\mathbf{A}_{i}, \Phi_i')$, is injective.
  Suppose $i,j \in I$ are such that $\alpha (i) = \alpha(j)$.
  Then $\Def\arii{m} (\mathbf{A}_{i}, \Phi_i') = \Def\arii{m} (\mathbf{A}_{j}, \Phi_j')$,
  and thus by
  Theorem~\ref{teor:fin_many_class_closed_under_unio_int_Sc},
  $\Def(\mathbf{A}_{i}, \Phi_i') = \Def(\mathbf{A}_{j}, \Phi_j')$, which implies
  $\Def(\mathbf{A}_{i}, \Phi_i) = \Def(\mathbf{A}_{j}, \Phi_j)$. Hence $\ab{A}_i$ and
  $\ab{A}_j$ are algebraically equivalent, which implies $i = j$.
  Thus $\alpha$ is injective, and therefore $I \le 2^{2^{|A|^{|A|^2}}}$.
 \end{proof}  
The set of first order formulas consisting of all quantifier free formulas is closed under
$\wedge$, $\vee$, and taking minors. Hence from Theorem~\ref{teor:fin_many_class_closed_under_unio_int_Sc},
we also obtain:
\begin{corollary}[{\cite[Theorem~1]{pinus17logeq}}]
The number of pairwise $L_0$-logically inequivalent algebras on a finite set $A$ is finite.
\end{corollary}

\section*{Acknowledgements}
A part of this work was done while the first listed author was visiting the
University of Siena. The authors thank Paolo Aglian\`o and Stefano
Fioravanti for their kind support. We also thank
Peter Mayr and
Tam\'{a}s Waldhauser for discussions on the topic of this
note.

\bibliographystyle{amsalpha}
\bibliography{airo11}
\end{document}